\documentclass{article}
\usepackage[utf8]{inputenc}

\usepackage{hyperref}
\usepackage[
    bibstyle=alphabetic,
    citestyle=alphabetic,
    maxnames=100, 
    doi=false,
    url = false,
    isbn=false,
    giveninits = true, 
]{biblatex}
\usepackage{verbatim}

\usepackage{amsthm}
\theoremstyle{plain}
\newtheorem{theorem}{Theorem}[section]
\newtheorem{proposition}[theorem]{Proposition}
\newtheorem{lemma}[theorem]{Lemma}
\theoremstyle{definition}
\newtheorem{definition}[theorem]{Definition}
\newtheorem{example}[theorem]{Example}
\newtheorem{problem}[theorem]{Problem}
\newtheorem{question}[theorem]{Question}
\newtheorem{conjecture}[theorem]{Conjecture}
\newtheorem{remark}[theorem]{Remark}

\usepackage{enumerate}

\usepackage{amsmath, amssymb}
\DeclareMathOperator{\Ind}{Ind}
\DeclareMathOperator{\gr}{gr}
\DeclareMathOperator{\ord}{ord}

\usepackage{xcolor}
\newcommand{\Nate}[1]{\texttt{\color{red}Nate: #1}}
\newcommand{\Josh}[1]{\texttt{\color{blue}Josh: #1}}

\title{Higher Congruences in Character Tables}
\author{Nate Harman, Joshua Mundinger}
\date{August 4, 2026}

\begin{document}

\maketitle

\begin{abstract}
Motivated by recent work of Peluse and Soundararajan on divisibility properties of the entries of the character tables of symmetric groups, we investigate the question: For a finite group $G$, when are two columns of the character table of $G$ congruent to one another modulo a power of a prime?

\end{abstract}

\section{Introduction}

In \cite{miller19} Miller made the following beautiful conjecture, which was proved in a pair of papers by Peluse and Soundararajan.

\begin{theorem}[\cite{PS22,PS25}]
For any fixed natural number $r$, the proportion of entries in the character table of $S_n$ which are divisible by $r$ tends to 1 as $n$ goes to infinity.
\end{theorem}

When $r$ is a prime number (or a product of distinct primes) this was generalized to wreath products $G \wr S_n$ when $G$ has rational character table and to Weyl groups of type $D_n$ in \cite{reu22}. 

In both of the Peluse and Soundararajan arguments, as well as the generalization to wreath products, a key representation-theoretic ingredient to the proof is a lemma which gives a criterion for when two columns of the character table of $S_n$ are congruent modulo a prime power $p^m$.
Let $\sim_m^{comb}$ denote the equivalence relation on partitions generated by $\lambda \sim_m^{comb} \mu$ if $\mu$ is formed from $\lambda$ by replacing $p^m$ parts of size $k$ with $p^{m-1}$ parts of size $kp$.  For $\sigma$ and $\sigma'$ in $S_n$, we say $\sigma \sim_m^{comb} \sigma'$ if the partitions corresponding to their cycle types are equivalent via $\sim_m^{comb}$.

\begin{lemma}[\cite{PS25}, Lemma 2.1] \label{lemma: PSlemma}
For $\sigma$, $\sigma'$ in $S_n$, if $\sigma \sim_m^{comb} \sigma'$, then $\chi(\sigma) \equiv \chi(\sigma') \mod p^m$ for all characters $\chi$ of $S_n$.
\end{lemma}

Having in mind the general goal of investigating character table divisibility questions as in Miller-Peluse-Soundararajan, the motivating question of this paper is: 
\begin{quote}
    \emph{For a finite group $G$, when are two columns of the character table of $G$ congruent to one another modulo $p^m$?}
\end{quote}

\begin{remark}

This question is reminiscent of, but distinct from, the theory of blocks in modular representation theory. Two characters are in the same $p$-block if they have the same central character modulo $p$ -- this means that the \emph{rows} of the central character table (a renormalization of the usual one) are congruent modulo $p$. For symmetric groups the Nakayama conjecture (now a theorem) says that the $p$-blocks are determined combinatorially by the $p$-core of a partition.  Our story is in a sense orthogonal to that in that we are interested in congruences between \emph{columns} of the character table.

\end{remark}

The remainder of the paper is structured as follows:

\subsubsection*{Section 2: Unramified Character Tables}  
We want to talk about entries of the character table of a finite group being congruent modulo a power of a prime $p$.  For the symmetric groups $S_n$, all of the character table entries are integers, so this makes sense automatically. However, for a general group $G$, the entries of the ordinary character table may not be rational integers and instead may live in some finite integral extension of $\mathbb{Z}$.  If this extension is ramified at $p$, then there may be no well-behaved notion of two entries being congruent modulo $p^m$. 


In §\ref{section: unramified table} we introduce the notion of the unramified character table of a group $G$ at a prime $p$, and prove some basic properties of them. Morally, the unramified character table is as close as you can get to working over $\mathbb{C}$, while still ensuring you have a well-behaved notion of congruence modulo $p^m$ for the entries.  We believe this is the ``right" setting for asking Miller-Peluse-Soundararajan-type divisibility questions about general groups.

In many natural cases of interest the unramified character table is the same as either the ordinary character table or the rational character table. So depending on the interests of the reader this section may be skipped.

\subsubsection*{Section 3: Higher Congruences}

A key step in the Peluse and Soundararajan argument is a combinatorial characterization of when two columns of the character table of $S_n$ are congruent to one another modulo a prime power $p^m$.  
When $m=1$ there is a general criterion essentially due to Brauer for when two columns of the character table are congruent modulo $p$.  However for $m>1$ the story seems to be much less well understood.


Our first main technical result is Proposition \ref{prop: upper bound by orthogonality}, which establishes an upper bound on the largest power of $p$ for which two columns of an unramified character table can be congruent without being forced to the same column. Proposition \ref{prop: upper bound for p-groups} gives an improvement to this bound in the important special case where the group is a $p$-group.

We then define an equivalence relation $\sim_m$ on a general group $G$, which is a generalization of the relation $\sim_m$ on $S_n$. Our second main technical result is Theorem \ref{theorem: prime-power-congruence} which says that if $g \sim_m g'$, then $\chi(g) \equiv \chi(g') \mod p^{m}$ for all unramified characters $\chi$ of $G$.

\subsubsection*{Section 4: Revisiting the Symmetric Group}

We then revisit the combinatorial criterion for when two columns of the character table of $S_n$ are congruent modulo a prime power $p^m$. We show that in this case $\sim_m$ indeed agrees with $\sim_m^{comb}$. Moreover we show in Proposition \ref{proposition: symmetric-group-congruence} that the converse of Lemma \ref{lemma: PSlemma} also holds: two columns of the character table of $S_n$ are congruent modulo $p^m$ if and only if the corresponding partitions are equivalent under $\sim_m^{comb}$ (as far as we can tell, only the ``if" direction is in the existing literature for $m > 1$).  

\subsubsection*{Section 5: Questions and Conjectures} 

Finally we close out by posing a number of related questions and conjectures.

\subsection*{Acknowledgments}  Thanks to Alexander Miller for helpful conversations.  We'd also like to thank our REU students Brandon Dong, Skye Rothstein, Lola Vescovo, Hannah Graff, Saskia Solotko, Devin Brown, Francisco Gutierrez, Rylan Rosas, Chris Wu, and Charles Troutman for inspiring these investigations. 

N.H.\ was partially supported by the National Science Foundation grant DMS-2401515. J.M.\ was partially supported by the NSF GRFP and by the NSF MSPRF Award No. 2503534. Any opinions, findings, and conclusions or recommendations expressed in this material are those of the authors and do not necessarily reflect the views of the National Science Foundation.

\section{The Unramified Character Table}\label{section: unramified table}

When are two columns of the character table congruent modulo $p^m$? 
For symmetric groups, all of the character values are integers, so it is clear what this question means.  
For a general finite group, characters live in the  ring of integers $\mathcal{O}$ of a number field $K$, and so the question needs refinement.

One natural attempt is to pick a prime ideal $\mathfrak{p} \subset \mathcal{O}$ lying over $p$ and ask whether two columns are congruent modulo $\mathfrak{p}^m$.
The answer to that question does not depend on the choice of prime ideal $\mathfrak{p}$: the Galois group $Gal(K/\mathbb{Q})$ acts on both the characters of $G$ and the prime ideals over $p$ compatibly\footnote{It is safe to assume $K/\mathbb Q$ is Galois, in which case the Galois group acts transitively on primes over $p$. In fact $K/\mathbb Q$ can be taken to be a cyclotomic extension by Brauer's theorem \cite[§12.3]{serre_linear-representations}.},
so if $\sigma \in Gal(K/\mathbb Q)$, then 
$$\chi(g) \equiv \chi(h) \mod \mathfrak{p}^m \iff \chi^{\sigma}(g) \equiv \chi^\sigma(h) \mod \sigma(\mathfrak{p})^m.$$

However, asking for congruence modulo $\mathfrak p^m$ is ill-behaved in the presence of ramification.
Consider for example $G = \mathbb Z/p \mathbb Z$, in which case the characters are valued in $K = \mathbb Q(\zeta_p)$ where $\zeta_p$ is a primitive $p$-th root of unity. The prime $p$ is totally ramified in $K$, and $(p) = \mathfrak p^{p-1}$ for $\mathfrak p=(\zeta_p-1)$.
In this case, one finds for all $g \neq h \in G$ and nontrivial simple $\chi$ that $\chi(g) \equiv \chi(h) \mod \mathfrak p$ but $\chi(g) \not\equiv \chi(h) \mod \mathfrak p^2$.
This is perhaps unexpected, since the Galois group of $\mathbb Q(\zeta_p)/\mathbb Q$ acting on the character table has two orbits, the identity and non-identity columns, while congruence treats the identity and non-identity columns the same.

We propose studying a slightly different notion of congruence. Rather than working with the ordinary character table, whose entries may be ramified at $p$, we work with a smaller character table whose values are guaranteed to be unramified. This is the \emph{unramified character table}.

\subsection{Unramified characters}

Fix a finite group $G$ and a prime $p$ dividing $|G|$.
Given a complex character $\chi$ of $G$,
let $\mathbb Q(\chi)$ be the field extension of $\mathbb Q$ generated by the values of $\chi$. 

\begin{definition}
    A character $\chi$ is \emph{unramified} at $p$ if the field extension $\mathbb Q(\chi)/\mathbb Q$ is unramified at $p$.
\end{definition}

Let $R_{p'}(G)$ be the ring of characters of $G$ unramified at $p$.
The ring $R_{p'}(G)$ is essentially determined by a cyclotomic Galois action. We summarize this procedure here.
By Brauer's induction theorem, $\mathbb Q(\chi) \subseteq \mathbb Q(e^{2\pi i/|G|})$ \cite[§10.5]{serre_linear-representations}.
Let $|G| = rp^e$ for $r$ prime to $p$, and let $\Gamma$ be the kernel of the projection $(\mathbb Z/|G|\mathbb Z)^\times \to (\mathbb Z/r\mathbb Z)^\times$. The group $\Gamma$ is isomorphic to $(\mathbb Z/p^e\mathbb Z)^\times$ by the Chinese Remainder Theorem,
and is the Galois group of $\mathbb Q(e^{2\pi i/|G|})/\mathbb Q(e^{2\pi i/r})$.
\begin{definition}
    A \emph{$\Gamma$-conjugacy class} of $G$ is an equivalence class of the relation on $G$ generated by conjugation and $x \sim x^s$ for $s \in \Gamma$.
\end{definition}    

\begin{proposition}[\cite{serre_linear-representations} §12.4, Corollary 1]\label{prop:  basis for gamma-class functions}
    \label{prop: character-conjugacy}
    A character $\chi$ is in $R_{p'}(G)$ if and only if $\chi$ is constant on $\Gamma$-conjugacy classes.
    The rank of $R_{p'}(G)$ is the number of $\Gamma$-conjugacy classes.
\end{proposition}

It follows from Proposition \ref{prop: character-conjugacy} that a basis for $R_{p'}(G)$ is given by sums of orbits of complex irreducible characters under the action of $\Gamma$ via $\chi^s(g) = \chi(g^s)$. These are the characters of the irreducible unramified representations of $G$, that is, those representations over an unramified extension of $\mathbb{Q}$ which remain irreducible under any unramified extension of the base field.
The \emph{unramified character table} is then the table of values of irreducible unramified characters on $\Gamma$-conjugacy classes.

\begin{remark}\label{remark: when is ramified = rational?}
    If $G$ has rational character table, then the unramified character table is the usual character table; if $G$ is a $p$-group, then the unramified character table is the rational character table.
\end{remark}

\begin{example}
    Let $G = C^{p^e} = \langle \sigma\rangle$ be a cyclic group of order a power of $p$.
    Then $\Gamma = (\mathbb Z/p^e\mathbb Z)^\times$.
    The $\Gamma$-conjugacy classes in $G$ are represented by $1,\sigma,\sigma^p,\ldots, \sigma^{p^{e-1}}$.
    The characters of $G$ are $\chi_\ell: \sigma \mapsto \zeta^\ell$ for $\zeta$ a primitive $p^{e}$th root of unity. A set of orbit representatives for $\Gamma$ acting on the characters of $G$ are $\chi_{p^0},\chi_{p^1},\ldots, \chi_{p^e} = \chi_0$; the unramified characters are sums over these orbits.
    For example, the unramified character table of $C^{p^4}$ is as below:
    \begin{center}
    \begin{tabular}{|c|ccccc|}\hline 
            & 1 & $\sigma$    & $\sigma^p$ & $\sigma^{p^2}$ & $\sigma^{p^3}$\\ \hline 
        $\chi_{p^4}$ & 1  & 1     & 1     &   1   & 1 \\ 
        $[\chi_{p^3}]$ & $p-1$ & $p-1$ & $p-1$    & $p-1$ & $-1$ \\ 
        $[\chi_{p^2}]$ & $p^2-p$ & $p^2-p$ & $p^2-p$ & $-p$ & 0 \\
        $[\chi_{p}]$ & $p^3 -p^2$ & $p^3 -p^2$ & $-p^2$ & 0 & 0\\
        $[\chi_{1}]$ &$p^4 - p^3$ & $-p^3$ & 0 & 0 & 0 \\ \hline
    \end{tabular}
    \end{center}
\end{example}

\begin{example}
        Let $V$ be an $n$-dimensional vector space over $\mathbb F_p$. Then the $\Gamma$-conjugacy classes in $V$ are exactly $0$ and $\mathbb P(V)$, the projectivization of $V$. After choosing an additive character $\psi: \mathbb F_p \hookrightarrow \mathbb C^\times$, the complex irreducible characters of $V$ are in bijection with the dual space $V^\ast$ via the pairing $V^\ast \times V \to \mathbb F_p \to \mathbb C^\times$. The $\Gamma$-orbits on complex characters are $\{0\} \cup \mathbb P(V^\ast)$, so the irreducible unramified characters are in bijection with $\{0\}\cup \mathbb P(V^\ast)$. The unramified character table is:
        \begin{center}
            \begin{tabular}{|c|c|c|}\hline 
            & 0 & $x \in \mathbb P(V)$ \\ \hline
            0 & 1 & 1 \\ \hline 
            $H \in \mathbb P(V^\ast)$ & $p-1$ & $
            \begin{cases} p-1 & x \in H \\ -1 & x \notin H\end{cases}$ \\ \hline 
            \end{tabular}
        \end{center}
\end{example}

\subsection{Realizability}

By definition, an unramified character $\chi$ has $\mathbb Q(\chi)$ unramified at $p$. This does not imply the character $\chi$ is realizable by a representation over $\mathbb Q(\chi)$. However, there is an unramified extension of $\mathbb Q_p(\chi)$ over which $\chi$ is realized. Since the values of a general character $\chi$ are algebraic integers, the field $\mathbb Q_p(\chi)$ makes sense, and $\chi$ is unramified at $p$ if and only if $\mathbb Q_p(\chi)/\mathbb Q_p$ is unramified.

\begin{proposition}\label{prop: realizability}
    If $\chi$ is a character of $G$ unramified at $p$,
    then there is a representation of $G$ with character $\chi$ defined over an unramified extension of $\mathbb Q_p(\chi)$.
\end{proposition}
\begin{proof}
    Let $L$ be the maximal unramified extension of $\mathbb Q_p$.
    By \cite[Chapter XII, Theorem 1]{serre_local-fields}, the Brauer group of $L$ is trivial.
    Hence by \cite[§12.2, Corollary]{serre_linear-representations}, every character of $G$ with values in $L$ is afforded by a representation defined over $L$.
    If $\chi$ is unramified at $p$, then $\mathbb Q(\chi) \subseteq L$.
\end{proof}

A nontrivial extension of $\mathbb Q_p(\chi)$ may be necessary, as in the next example: 
\begin{example}
    Consider $G = \{\pm 1, \pm i, \pm j, \pm k\}$, the order 8 group of unit quaternions.
    The group $G$ has one absolutely irreducible character of degree greater than 1, given by $\chi(\pm 1) = \pm 2$ and $\chi(\pm i) = \chi(\pm j) = \chi(\pm k) = 0$.
    Since $\chi$ is integer-valued, $\chi$ is unramified at all primes.
    To construct a representation affording this character, consider the quaternion algebra
    \[ A= \left( \frac{-1,-1}{\mathbb Q}\right).\] 
    Then $G \subseteq A$, and the induced map $\mathbb Q G \to A$ is the simple factor of $\mathbb QG$ corresponding to the character $\chi$. 
    Thus, for a field $K/\mathbb Q$, $\chi$ is afforded by a $KG$-module if and only if $A \otimes_{\mathbb Q} K$ splits.
    
    It is well-known that $A\otimes_\mathbb Q K$ splits if and only if $-1=  a^2 + b^2$ for $a,b \in K$ (see  \cite[Exercise 12.3]{serre_linear-representations}).
    Thus $\chi$ is not realizable over $\mathbb Q_2 = \mathbb Q_2(\chi)$, since $-1$ is not a sum of squares in $\mathbb Q_2$.
    However, the number field $K = \mathbb Q(\zeta_3)$ for $\zeta_3$ a primitive $3$rd root of unity is unramified at 2; 
    since $-1 = \zeta_3^2 + \zeta_3 = \zeta_3^2 + (\zeta_3^2)^2,$
    the character $\chi$ is afforded by a $\mathbb Q(\zeta_3)G$-module.
\end{example}

\section{Higher congruences}


When are two columns of the unramified character table of $G$ congruent modulo some power of $p$?
The question of congruence modulo $p$ is entirely addressed by Brauer's work on modular characters. Two elements are congruent mod $p$ for all unramified characters if and only if they agree on all Brauer characters. If $g = g_sg_u$ is the $p$-semisimple and $p$-unipotent decomposition of $g\in G$, then $g$ and $g'$ are equal on all Brauer characters if and only if $g_s$ and $g'_s$ are conjugate \cite[§18.1]{serre_linear-representations}. 

This answers the question of when two columns are congruent modulo $p$. On the other hand, if two columns are congruent modulo a sufficiently large power of $p$, then they are actually equal: 

\begin{proposition}\label{prop: upper bound by orthogonality}
    Let $e= \ord_p|G|$, and let $f$ be the maximum of $\ord_p \langle \chi,\chi\rangle$ as $\chi$ ranges over unramified irreducible characters of $G$.
    If $g_1,g_2 \in G$ are such that $\chi(g_1) \equiv \chi(g_2) \mod p^{e+f+1}$ for all unramified irreducible characters $\chi$, then $g_1$ and $g_2$ are $\Gamma$-conjugate.
\end{proposition}
\begin{proof}
    We proceed by column orthogonality of the unramified character table.
    Let $C_\Gamma(g)$ be the $\Gamma$-conjugacy class of $g\in G$, and let $\delta_{C_\Gamma(g)}$ be the $\Gamma$-class function 
    \[\delta_{C_\Gamma(g)}(h) = \begin{cases} 1 & h \sim_\Gamma g \\ 0 & h \not\sim_\Gamma g.\end{cases}\]
    By Proposition \ref{prop:  basis for gamma-class functions}, the unramified irreducible characters form a basis for $\Gamma$-class functions, so
    $\delta_{C_\Gamma(g)} = \sum_\chi \left(\langle \delta_{C_\Gamma(g)},\chi\rangle/\langle \chi,\chi\rangle\right) \chi$,
    where the sum is over unramified irreducible characters. By definition $\langle \delta_{C_\Gamma(g)},\chi\rangle = |C_\Gamma(g)|\chi(g^{-1})/|G|$. Thus for $g_1,g_2 \in G$,
    \[ 
    \sum_{\chi \text{ unr. irr.}} \frac{\chi(g_1)\chi(g_2^{-1})}{\langle\chi,\chi\rangle} = 
    \begin{cases}
        |G|/|C_\Gamma(g_1)| & g_1 \sim_\Gamma g_2 \\ 
        0 & g_1 \not\sim_\Gamma g_2.
    \end{cases}
    \]
    Suppose $g_1,g_2 \in G$ are not $\Gamma$-conjugate. Let $M$ be the least common multiple of $\langle\chi,\chi\rangle$ over the unramified irreducible characters $\chi$. Then $\sum_\chi \chi(g_1)\chi(g_2^{-1})/\langle\chi,\chi\rangle = 0$, while 
    \[ \sum_{\chi\text{ unr. irr.}} \frac{M}{\langle \chi,\chi\rangle} \chi(g_1)\chi(g_1^{-1}) = \frac{M|G|}{|C_\Gamma(g_1)|}.\]
    The left-hand side is an integer combination of unramified character values, while the right-hand side has $p$-adic valuation at most $\ord_p M + \ord_p|G| = e + f$. If $\chi(g_1)\equiv \chi(g_2) \mod p^{e+f+1}$ for all $\chi$, then we obtain
    \[ 0 \equiv \frac{M|G|}{|C_\Gamma(g_1)|} \mod p^{e+f+1},\]
    a contradiction. Hence $g_1$ and $g_2$ are $\Gamma$-conjugate, as desired.
\end{proof}
This bound works best when the complex character values of $G$ are rational, or more generally, tamely ramified over $\mathbb Q$, in which case $f=0$. If $G$ is a $p$-group, then the character values may be wildly ramified. For example, if $G = C_{p^e}$, then $f = e-1$, and Proposition \ref{prop: upper bound by orthogonality} states that distinct $\Gamma$-conjugacy classes can be congruent to order at most $p^{2e-1}$. This is much worse than necessary. For a $p$-group, we have a better bound, which is tight for $C_{p^e}$:
\begin{proposition}\label{prop: upper bound for p-groups}
    If $G$ is a $p$-group of order $p^e$ and $g_1,g_2 \in G$ are such that $\chi(g_1)\equiv \chi(g_2) \mod p^{e+1}$ for all unramified characters of $G$, then $g_1$ and $g_2$ are $\Gamma$-conjugate.
\end{proposition}
\begin{proof}
    If $G$ is a $p$-group, then $g_1$ and $g_2$ are $\Gamma$-conjugate if and only if $g_1$ and $g_2$ generate conjugate cyclic subgroups. So let $C$ be the cyclic subgroup generated by $g_1$.
    Consider $\chi_1 = \Ind_C^G 1$. Then $\chi_1(g_1) = [N_G(C):C]$, so $\ord_p(\chi_1(g_1)) \leq e$,
    while $\chi_1(g) = 0$ if $g$ is not conjugate to an element of $C$.
    As $\chi_1(g_2) \equiv \chi_1(g_1) \neq 0 \mod p^{e+1}$,
    $\chi_1(g_2)$ is not zero.
    Hence, $g_2$ is conjugate into $C$.
    
    Now let $\chi_2 = \Ind_{C^p}^G 1$.
    The same argument shows $\ord_p(\chi_2(g_2)) \leq e$ if $g_2$ is conjugate into $C^p$, while $\chi_2(g_1) = 0$. Hence $g_2$ is not conjugate to an element of $C^p$. Since $C$ is a $p$-group, this implies $g_2$ is conjugate to a generator of $C$, so $g_1$ and $g_2$ generate conjugate cyclic subgroups of $G$.
\end{proof}

In between these two extremes, it is more difficult to determine the order of congruence between two columns of the unramified character table. We present a sufficient criterion for two columns to be congruent modulo some power $p^m$.
Motivation comes from the following reframing of Brauer's condition: 
consider the relation $\sim$ on $G$ generated by conjugacy and the relation $g \sim gh$ if $h$ has order a power of $p$ and commutes with $g$.
If $g, g' \in G$, the $p$-semisimple parts $g_s$ and $g'_s$ are conjugate if and only if $g \sim g'$.
We will show in Theorem \ref{theorem: prime-power-congruence} that if $h$ as above is a $p$-th power in $G$, then a deeper congruence between $g$ and $gh$ is obtained. 

To prove this theorem, we first deal with the case when $g$ and $h$ are both of order a power of $p$, the ``totally ramified'' case.
\begin{lemma} \label{lemma: unipotent-congruence}
    Let $K$ be a finite extension of $\mathbb Q$ or $\mathbb Q_p$ unramified over $p$. If $V$ is a finite-dimensional vector space over $K$ and $A,B\in GL(V)$ are each of order a power of $p$ and satisfy $AB = BA$, then 
    \[ tr(AB^{p^{m-1}}) \equiv tr(A) \mod p^{m}.\]
\end{lemma}
\begin{proof}
    Since we are taking a trace, we may replace $V$ with its semisimplification under the action of $K[A,B]$.
    By the Nullstellensatz, the simple $K[A,B]$-modules are of the form $L$ for $L/K$ a finite extension generated by $A,B\in L$.
    Since $A$ and $B$ have order a power of $p$, $L = K(\zeta)$ for some primitive $p^j$th root of unity $\zeta$, and $A$ and $B$ are both powers of $\zeta$.
    Since $p$ is unramified in $K$, the trace form for $\xi \in \langle \zeta\rangle$ is given by
    \[ tr_{L/K}(\xi) = \begin{cases}
        p^j - p^{j-1} & \xi = 1 \\
        -p^{j-1} & \xi^p = 1, \xi \neq 1 \\
        0 & \text{else}
    \end{cases}
    \] 
    so that $tr_{L/K} = \chi_1 - \chi_2$ for $\chi_1$ the regular character of $\langle \zeta\rangle$ and $\chi_2$ the regular character of $\langle \zeta\rangle/\langle \zeta^{p^{j-1}}\rangle$.
    
    Thus, it suffices to show that if $C$ is a cyclic $p$-group with regular character $\chi$ and $A,B \in C$, then $\chi(AB^{p^{m-1}}) \equiv \chi(A) \mod p^{m}$.
    The only case when $\chi(AB^{p^{m-1}}) \neq \chi(A)$ is if exactly one of $AB^{p^{m-1}}$ and $A$ is not the identity. Hence $B^{p^{m-1}}$ is not the identity, so $C$ must have order greater than $p^{m-1}$. Thus $p^{m} \mid |C|$. But the nonzero value of $\chi$ is $|C|$. Hence $\chi(AB^{p^{m-1}}) \equiv \chi(A) \mod p^{m}$, establishing the claim.
\end{proof}

With the totally ramified case in hand, the general case follows by factoring into unramified and totally ramified cases:

\begin{theorem}\label{theorem: prime-power-congruence}
    Suppose that $g, h \in G$, $h$ commutes with $g$, and $h$ has order a power of $p$.
    Then for all unramified characters $\chi$ of $G$,
    \[ \chi(g) \equiv \chi(gh^{p^{m-1}}) \mod p^{m}.\]
\end{theorem}

\begin{proof}
    Let $L = \mathbb Q^{ur}_p$ be the maximal unramified extension of $\mathbb Q_p$.
    By Proposition \ref{prop: realizability}, there is a $LG$-module $V$ affording $\chi$.
    Consider the radical filtration on $V$ defined by the action of $g$, and let $\gr(V)$ be the associated graded. 
    Since $h$ commutes with $g$, the action of $h$ on $V$ induces an action of $h$ on $\gr(V)$.
    
    Let $g = g_sg_u$ be the $p$-semisimple and $p$-unipotent decomposition of $g$.
    Since $g_s$ has order prime to $p$, all its eigenvalues lie in $L$, so $\gr(V)$ has a weight space decomposition with respect to the semisimple operator $g_s$. 
    Let $W$ be a weight space for $g_s$.
    Since $g_s$ and $g_u$ are powers of $g$,
    $h$ commutes with $g_s$ and $g_u$. Hence $g_u$ and $h$ are commuting operators on $W$.
    By Lemma \ref{lemma: unipotent-congruence}, $tr(g_u|W) \equiv tr(g_uh^{p^{m-1}}|W) \mod p^{m}$.
    Since $g_s$ is acting by a scalar on $W$, $tr(g|W) \equiv tr(gh^{p^{m-1}}|W) \mod p^{m}$.
    Summing over all weights of $g_s$ gives the desired result.
\end{proof}

\begin{definition} \label{definition: prime-power-relation}
    Given a group $G$ and $m \geq 1$, let $\sim_{m}$ be the equivalence relation on $G$ generated by conjugation and $g \sim_m gh^{p^{m-1}}$ whenever $h$ commutes with $g$ and has order a power of $p$.
\end{definition}

Theorem \ref{theorem: prime-power-congruence} shows that if $g \sim_m g'$, then $\chi(g) \equiv \chi(g') \mod p^{m}$ for all $\chi$. However, the converse does not hold for every group.

\begin{example}  \label{heisenbergex}
    Let $H$ be the Heisenberg group of order $p^{2n+1}$. If $(V,\omega)$ is a $2n$-dimensional symplectic vector space over $\mathbb F_p$, then $H$ is the central extension 
    \[ 0 \to Z=\mathbb F_p \to H \to V \to 0\]
    given by the 2-cocycle $\omega$.
    The complex characters with nontrivial central character are of the form $\psi = \frac{1}{p^n} \Ind_Z^H \eta$ for $\eta$ a nontrivial character of the cyclic group $Z$. Hence, there is a single irreducible unramified representation with nontrivial central action, namely $\psi^{nr} = \frac{1}{p^n}\Ind^H_Z IZ$ for $IZ$ the augmentation ideal of $Z$. For $z \in Z$, we then have $\psi^{nr}(z) = p^n(-1)$, while $\psi^{nr}(1) = p^n(p-1)$. Since all other unramified irreducible representations have trivial central action, we obtain that $\chi(z) \equiv \chi(1)\mod p^n$ for all unramified characters $\chi$ of $H$.
    
    However, all elements in $H$ are of order $p$, so there are no non-trivial $p$th powers. Hence, no nontrivial higher congruences are explained by Theorem \ref{theorem: prime-power-congruence}.
\end{example}

\begin{remark}
    It is known that if $B$ is a square integer matrix, then $tr(B^{p^{m-1}}) \equiv tr(B^{p^m}) \mod p^m$ \cite{Ste17}, which for $m=1$ is a generalization of Fermat's little theorem to matrices.
    Thus, if $B$ is an integer matrix whose order is a power of $p$, then inductively applying $tr(B^{p^{i-1}}) \equiv tr(B^{p^i}) \mod p^i$ for $i \geq m$ implies that $tr(B^{p^{m-1}}) \equiv tr(1)\mod p^m$, which is the conclusion of Lemma \ref{lemma: unipotent-congruence} when $A=1$.
    It would be interesting to understand more connections between these congruences.
\end{remark}

\section{Higher congruences in the symmetric group}

While in general, the relation $\sim_m$ of Definition \ref{definition: prime-power-relation} does not always coincide with congruence modulo $p^m$ for $m > 1$,
we will show that in the case of the symmetric group $S_n$ that $\sim_m$ does coincide with congruence modulo $p^m$ for all $m$. To show this, we give a combinatorial characterization of the relation $\sim_{m}$ on elements of the symmetric group. 

Recall that for $\sigma \in S_n$, the \emph{cycle type} of $\sigma$ is the partition of $n$ formed by the lengths of cycles of $\sigma$.
The conjugacy classes of $S_n$ are in bijection with partitions of $n$ via cycle type. 
\begin{definition}
    The relation $\sim_m^{comb}$ on partitions is the equivalence relation generated by $\lambda \sim_m^{comb} \mu$ if $\mu$ is formed from $\lambda$ by replacing $p^m$ parts of size $k$ with $p^{m-1}$ parts of size $kp$.
\end{definition}
Let $\sim_m^{comb}$ also denote the relation on conjugacy classes of $S_n$ defined this way.
\begin{proposition}\label{proposition: symmetric-group-congruence}
    For $\sigma,\sigma' \in S_n$, the following are equivalent:
    \begin{enumerate}[i)]
        \item $\sigma \sim_m^{comb} \sigma'$;
        \item $\sigma \sim_m \sigma'$;
        \item $\chi(\sigma) \equiv \chi(\sigma') \mod p^m$ for all characters $\chi$ of $S_n$.
    \end{enumerate}
\end{proposition}

To prove Proposition \ref{proposition: symmetric-group-congruence}, we need to construct certain representations of the symmetric group. 
Given a partition $\lambda \vdash n$, let $S_\lambda = S_{\lambda_1}\times S_{\lambda_2}\times\cdots$ be the Young subgroup associated to $\lambda$, and let $M^\lambda$ be the character $Ind_{S_\lambda}^{S_n} 1$ induced from the trivial character of $S_\lambda$.
It is a standard fact that $\{M^\lambda\}_{\lambda \vdash n}$ is a basis for the representation ring of $S_n$ \cite[Theorem 2.2.10]{james2006representation},
although we do not need it. Given $\mu \vdash n$, let $M^\lambda_\mu$ denote the value of the character $M^\lambda$ at the conjugacy class with cycle type $\mu$. We give a combinatorial interpretation of the character values $M^\lambda_\mu$:

\begin{definition}
    For $\lambda,\mu \vdash n$, a \emph{row decomposition} of $\lambda$ by $\mu$ is a function $f:\{\text{rows of }\mu\} \to \{\text{rows of }\lambda\}$ such that the preimage of a row of length $k$ is a set of rows with total length $k$.
    Let $RD(\lambda,\mu)$ be the set of all row decompositions of $\lambda$ by $\mu$.
\end{definition}

The following is then immediate from the definition of $M^\lambda$:
\begin{lemma}\label{lemma: row-decomposition-character-formula}
    If $\lambda,\mu \vdash n$, then $M^\lambda_\mu = |RD(\lambda,\mu)|$.
\end{lemma}

\begin{lemma}\label{lemma: axehead-congruence}
    Let $\mu = (\xi,k^{p^m})$ and $\nu = (\xi, (pk)^{p^{m-1}})$ be partitions of $n$ for positive integers $k,m$ and a fixed partition $\xi \vdash n-kp^m$. Then for $\lambda = (n-pk,k^p)$ we have 
    \[ M^\lambda_\mu \equiv M^\lambda_\nu - p^m \mod p^{m+1}\]
\end{lemma}
\begin{proof}
    Given $\mu$ and $\nu$ as in the statement, let 
    \begin{align*}
        R_1 &= \{\rho \in RD(\lambda,\mu) \mid \text{no rows of }k^{p^m}\text{ are placed in }k^p\text{ in }\rho\}, \\
        R_2 &= \{\rho \in RD(\lambda,\mu) \mid \text{between }1 \text{ and }p-1\text{ rows of }k^{p^m}\text{ are placed in }k^p\text{ in }\rho\}, \\
        R_3 &= \{\rho \in RD(\lambda,\mu) \mid p\text{ rows of }k^{p^m}\text{ are placed in }k^p\text{ in }\rho\}.
    \end{align*}
    Observe that every $\rho \in RD(\lambda,\nu)$ has all rows of $(pk)^{p^{m-1}}$ placed in the row of length $n-pk$. This gives a bijection $R_1 \cong RD(\lambda,\nu)$. Also, there are $(p^m)!/(p^m-p)!$ ways of choosing $p$ out of $p^m$ rows with some order,
    and given such a choice, tiling the rows $k^p$ accordingly gives a unique element of $R_3$. It follows that 
    \[|R_3| = (p^m)!/(p^m-p)! = p^m(p^m-1)\cdots(p^m-p+1).\]
    By Wilson's theorem, $(p-1)! \equiv -1 \mod p$, so $|R_3| \equiv -p^{m} \mod p^{m+1}$.
    
    By Lemma \ref{lemma: row-decomposition-character-formula}, $|R_1| + |R_2| + |R_3| = |RD(\lambda,\mu)|$ and $|R_1| = |RD(\lambda,\nu)|$. 
    We have shown $|R_3| \equiv -p^m \mod p^{m+1}$; thus, we must show $|R_2| \equiv 0 \mod p^{m+1}$. 
    
    We show $p^{m+1}$ divides $|R_2|$ by considering orbits of a group and its $p$-Sylow subgroup acting on $R_2$.
    There is an action of the symmetric group $S_p$ on $RD(\lambda,\mu)$ by permuting the $p$ rows tiled into $k^p$. There is a commuting action of $S_{p^m}$ on $RD(\lambda,\mu)$ by permuting the placements of the rows of length $k$ of $\mu$ not in $\xi$. Given $\rho \in R_2$, let $Q$ be the $p$-Sylow subgroup of the stabilizer of $\rho$ in $S_p \times S_{p^m}$. We will analyze the size of $Q$. 
    
    Let $\pi_1: Q \to S_p$ and $\pi_2: Q \to S_{p^m}$ be the projections onto the two direct factors of $S_p \times S_{p^m}$. If $\sigma \in Q$, then since $\rho$ has between $1$ and $p-1$ rows of $k^{p^m}$ placed into $k^p$, $\pi_1(\sigma)$ does not act transitively on the $p$ rows in $k^p$. Since $Q$ is a $p$-group, $\pi_1(\sigma)$ is either the identity or the $p$-cycle; since $\pi_1(\sigma)$ does not act transitively, $\pi_1(\sigma) = 1$. Thus $\pi_1$ is trivial and $\pi_2$ is injective.
    
    If $\rho$ has exactly $\ell$ rows of $k^{p^m}$ placed into $k^p$, then $\pi_2(\sigma)$ must fix those $\ell$ rows. Hence $\pi_2(Q) \subseteq S_{p^m -\ell}$. Now $p^m$ divides $[S_{p^m}: S_{p^m - \ell}]$. It follows that $p^{m+1}$ divides $[S_p \times S_{p^m}: Q]$. As the index of $Q$ in the stabilizer of $\rho$ is prime to $p$, we conclude $p^{m+1}$ divides the cardinality of the orbit of $\rho$. 
    
    Since $p^{m+1}$ divides the size of every orbit of $S_{p} \times S_{p^m}$ on $R_2$, it follows that $p^{m+1}$ divides $|R_2|$. This proves the claim.
\end{proof}

\begin{proof}[Proof of Proposition \ref{proposition: symmetric-group-congruence}]
    
 i) $\implies$ ii): suppose that $\sigma \sim_m^{comb} \sigma'$. 
    We may assume that the cycle type of $\sigma'$ is formed from that of $\sigma$ by replacing $p^m$ cycles of length $k$ with $p^{m-1}$ cycles of length $pk$.
    Write $\sigma = \sigma_1\sigma_2$ where $\sigma_1$ is a product of $p^m$ disjoint $k$-cycles and $\sigma_2$ is disjoint from $\sigma_1$. Then $\sigma_1 = \tau^{p^m}$ for $\tau$ a $p^mk$-cycle also disjoint from $\sigma_2$. 
    If $k = p^i \cdot \ell$ where $\ell$ and $p$ are relatively prime,
    then the cyclic group $C = \langle \tau \rangle$ is isomorphic to $\mathbb{Z}/p^mk\mathbb{Z} \cong \mathbb{Z}/p^{m+i}\mathbb{Z} \times \mathbb{Z}/\ell\mathbb{Z}$,
    and $\sigma_1$ corresponds to $(p^m,p^m) \in \mathbb{Z}/p^{m+i}\mathbb{Z} \times \mathbb{Z}/\ell\mathbb{Z}$.

   Observe that $(a,b) \in \mathbb{Z}/p^{m+i}\mathbb{Z} \times \mathbb{Z}/\ell\mathbb{Z}$ has order a power of $p$ if and only if $b=0$, and it is a $p^{m-1}$th power of an element if and only if $a$ is divisible by $p^{m-1}$.  Moreover, under the identification $\langle \tau \rangle \cong \mathbb{Z}/p^{m+i}\mathbb{Z} \times \mathbb{Z}/\ell\mathbb{Z}$, an element $(a,b)$ corresponds to a product of $p^{m-1}$ disjoint $pk$-cycles exactly when it has order $pk$, which happens if and only if $a$ is divisible by $p^{m-1}$ but not $p^m$ and $b$ is relatively prime to $\ell$.  So we see that $(p^m, p^m) = (p^{m-1}, 0) + ((p-1)p^{m-1}, p^m)$ corresponds to a decomposition of $\sigma_1$ as a $p^{m-1}$th power of an element of order a power of $p$, times a product of $p^{m-1}$ disjoint $pk$-cycles. By construction, this decomposition consists of elements commuting with $\sigma$. Thus $\sigma \sim_m \sigma'$.

    ii) $\implies$ iii): this is Theorem \ref{theorem: prime-power-congruence}.
    
    iii) $\implies$ i):
    We proceed by double induction: on $m$, and backwards on the length of the smallest cycle where $\sigma$ and $\sigma'$ differ in cycle type. The base case $m=1$ follows from Brauer's theory: if $\sigma$ has cycle type $\mu$, then the semisimple part $\sigma_s$ has cycle type formed by replacing parts of size $p^\ell k, (k,p)=1$, in $\mu$ with $p^\ell$ parts of size $k$. Hence, if $\sigma_s$ and $\sigma'_s$ are conjugate, then $\sigma \sim_1^{comb} \sigma'$.
    
    Now suppose that $m \geq 2$ and $\sigma,\sigma' \in S_n$ are such that $\chi(\sigma) \equiv \chi(\sigma') \mod p^m$ for all characters $\chi$ of $S_n$. Let $k$ be the smallest cycle length which appears in $\sigma$ and $\sigma'$ with different multiplicity. Then $\sigma$ and $\sigma'$ have cycle types $\mu = (\xi,k^i,\eta)$ and $\nu = (\xi,k^j,\tau)$ for partitions $\xi$, $\eta$, and $\tau$ such that all parts in $\eta$ and $\tau$ have size greater than $k$. By induction, we have $\sigma \sim_{m-1}^{comb} \sigma'$, which implies $p^{m-1} \mid (i-j)$. This implies $n \geq pk$, so we may set $\lambda = (n-pk,k^p)$. Any row decomposition of $\lambda$ by $\mu$ or $\nu$ places no parts of $\eta$ or $\tau$ into the parts of length $k$. Hence applying Lemma \ref{lemma: axehead-congruence} $(i-j)/p^{m-1}$ times gives
    \[ M^\lambda_\mu = M^\lambda_\nu - (i-j) \mod p^{m}.\]
    Since $\chi(\sigma) \equiv \chi(\sigma') \mod p^m$ for all characters $\chi$ of $S_n$, we conclude $p^m \mid (i-j)$. Assume without loss of generality that $i > j$. Then $\sigma' \sim_m^{comb} \sigma''$ for $\sigma''$ of cycle type $(\xi,k^i, (pk)^{(i-j)/p},\tau)$. By the case i) $\implies$ iii) above, we have $\chi(\sigma') \equiv \chi(\sigma'') \mod p^{m}$ for all characters $\chi$ of $S_n$. Now $\sigma$ and $\sigma''$ have the same number of cycles of length $k$, so by induction, $\sigma \sim_m^{comb} \sigma''$ and thus $\sigma \sim_m^{comb} \sigma'$.
\end{proof}

\section{Questions and Conjectures}

Proposition \ref{proposition: symmetric-group-congruence} says that $\sim_m$ fully characterizes higher congruences in the character table of the symmetric group, and based on our limited calculations it seems like this is often the case.
However as we saw in Example \ref{heisenbergex} this is not the case in general -- it seems to fail sometimes for $p$-groups (and their close relatives).  It is not clear to us why this is the case, so we will close the paper out with a few related questions and conjectures:

\begin{problem}\label{problem: single-prime-congruence}
    Given a fixed prime $p$, characterize the finite groups $G$ such that $\chi(g) \equiv \chi(g') \mod p^m$ for all unramified $\chi$ if and only if $g \sim_m g'$ as in Definition \ref{definition: prime-power-relation}.
\end{problem}

\begin{problem}\label{problem: all-prime-congruence}
    Characterize the finite groups $G$ such that for all primes $p$, $\chi(g) \equiv \chi(g') \mod p^m$ for all unramified $\chi$ if and only if $g \sim_m g'$ as in Definition \ref{definition: prime-power-relation}.
\end{problem}

\begin{conjecture}
    If $G$ is a finite simple group of Lie type and $p$ is a prime different from the defining characteristic of $G$, 
    then
    $\chi(g) \equiv \chi(g') \mod p^m$ for all unramified $\chi$ if and only if $g \sim_m g'$.
\end{conjecture}

\begin{question}\label{question: p-group-depth-of-congruence}
Let $G$ be a $p$-group.  For $g \in G$, what is the largest power $p^m$ such that $\chi(g) \equiv \chi(e) \mod p^m$ for all unramified characters $\chi$?
\end{question}

\begin{remark}
We could ask the same question for an arbitrary group $G$, but the $p$-group case is of particular interest. Moreover, we have reason to believe the $p$-group case might have a relatively nice answer. The unramified characters of $p$-groups are the same as the rational characters, and a result of Ford (\cite{ford87}, Theorem 3) says that every rational character of a $p$-group can be expressed as a difference of two permutation characters.  This suggests there is a group-theoretic answer to Question \ref{question: p-group-depth-of-congruence} in terms of subgroups and conjugacy inside $G$, although thus far we have been unable to find one.
\end{remark}

\printbibliography

\end{document}